\documentclass[a4paper,12pt]{article}
\usepackage[utf8]{inputenc}
\usepackage{amsmath,amsthm,amssymb,amsfonts}
\usepackage{bbm,bm}
\usepackage{latexsym}
\usepackage{mathrsfs}
\usepackage{threeparttable}
\usepackage{tabularx}
\usepackage{booktabs}
\usepackage{graphicx,subfigure}
\usepackage{xcolor}
\usepackage{indentfirst}
\usepackage[colorlinks,
            linkcolor=blue,
            citecolor=blue
            ]{hyperref}
\usepackage{geometry}
\usepackage{caption}
\usepackage{float}
\usepackage[numbers,sort&compress]{natbib}
\usepackage{epstopdf}
\usepackage{tikz}

\numberwithin{equation}{section}
\def \[{\begin{equation}}
\def \]{\end{equation}}

\newtheorem{thm}{Theorem}[section]

\newtheorem{defi}[thm]{Definition}

\newtheorem{lem}[thm]{Lemma}
\newtheorem{cor}[thm]{Corollary}
\newtheorem{ex}[thm]{Example}

\newtheorem{rem}[thm]{Remark}

\begin{document}

\synctex=1

\setlength{\baselineskip}{20pt}
\begin{center}{\Large  The resonance graphs of nanotubes and toroidal polyhexes}
\footnote{This work is supported by the National Natural Science Foundation of China (Grant Nos. 12271229 and 12671409).}

\vspace{4mm}

{Lingmei Liang, Heping Zhang\footnote{The corresponding author.}
\renewcommand\thefootnote{}\footnote{E-mail addresses:
lianglm2023@lzu.edu.cn (L.Liang), zhanghp@lzu.edu.cn (H.Zhang).}}

\vspace{2mm}

\footnotesize{ School of Mathematics and Statistics,
Lanzhou University, Lanzhou, Gansu 730000, P. R. China}

\end{center}

\noindent {\bf Abstract}: Coronoid systems, nanotubes and toroidal polyhexes (or fullerenes) can all be regarded as carbon networks composed of carbon atoms linked in hexagonal
shapes. The resonance graphs of coronoid systems and nanotubes are not necessarily connected. For coronoid systems and  elementary nanotubes, by using flow across cuts the present authors gave criteria  for two perfect matchings  lying in the same connected component of the resonance graph (Discrete Appl. Math. 395 (2026) 443-455).
However, the sufficiency of such criterion does not hold for general nanotubes and toroidal polyhexes. 
In this paper we  strengthen this requirement to obtain valid criteria  for two perfect matchings of a nanotube (resp. toroidal polyhex) to lie in the same connected component of its resonance graph: they have the same flows across cuts along the $x$-axis (resp. longitude and latitude) and the same ladders. For toroidal polyhexes, our method uses homotopic classes of simple loops on the  torus,
and the above criterion can be simplified by using only simple flows,  for the case in which two perfect matchings have alternating hexagons.

\vspace{2mm}
\vspace{1mm}


{\setcounter{section}{0}
\section{Introduction}\setcounter{equation}{0}

Perfect matchings of a graph may serve some important  models of the real world, such as  Kekul\'{e} structures of polycyclic aromatic compounds and domino tilings of grids. Transforming a perfect matching of a graph on a surface into another one by a flip along an alternating face with respect to this  perfect matching. This flip can establish a graph structure on the set of all perfect matchings, called resonance graph \cite{WGrun82}, $Z$-transformation graph \cite{ZGC88} or flip graph \cite{STCR95}.

The resonance graph was independently introduced by Gr$\rm{\ddot{u}}$ndler \cite{WGrun82} and   Zhang et al. \cite{ZGC88}  for hexagonal systems, and  extended to  plane bipartite graphs \cite{ZZ2000,Four03,Che18}, plane nonbipartite graphs \cite{LWLZ26, TZra16}, graphs on surfaces \cite{TYra23}.
Resonance graphs display rich structures which are related to  distributive lattices \cite{LZhang03,TZ16}, median graphs \cite{ZLs08}, daisy cubes \cite{BCTZ25} and metallic cubes \cite{DP26}. For an early survey, see \cite{Zhang06}.

It is well-known \cite{ZGC88,T90} that the resonance graphs of hexagonal and square lattices on the plane via flips along hexagons and squares  respectively are connected.
Liu et al.  showed that the  resonance graph of a cylindrical grid $P_{2m}\square C_{2n+1}$ is connected \cite{LWLZ26}, and     the resonance graph of a toroidal grid $C_{2m}\square C_{2n+1}$  consists of two isomorphic components \cite{LZZI25}. However,  the resonance graph of a nanotube  via   hexagon flips is not necessarily connected   \cite{LZhan26}. We  can see that a toroidal polyhex (fullerene) has the disconnected resonance graph:
The three direction edges form three perfect matchings as singleton components \cite{SLZh05} since they have no alternating hexagons (For example, see Fig. \ref{Figure2-1}).
\begin{figure}
\centering
\includegraphics{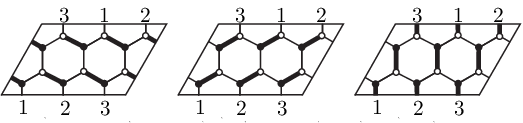}
\caption{\label{Figure2-1} Three perfect matchings of $H(3,2,1)$ have no alternating hexagons.}
\end{figure}

A natural problem is to give a criterion  when any given two perfect matchings of a graph on a surface  lie on the same connected component of its resonance graph.
Saldanha et al. \cite{STCR95} used homology and cohomology theory to obtain three criteria for  two domino tilings of a quadriculated region with holes  are in the same component of its flip graph.  The present authors \cite{LZhan26} showed that the combinatorial version in terms of flows across cuts holds for coronoids (hexagonal systems with holes) and elementary nanotubes by using purely graph-theoretical approach.
\begin{thm}[\cite{LZhan26}]{\label{EN-flow}}
Let $N$ be an elementary nanotube.
Two perfect matchings lie in the same connected component of the resonance graph of $N$  if and only if their flows across any cut segment $L$ are equal.
\end{thm}
 However, this criterion does not hold for a general nanotube (see  Example 4.10 in \cite{LZhan26}).
In this article we give simple criteria  for two perfect matchings of a nanotube and a toroidal polyhex to belong to the same connected component of their  resonance graphs (Theorems \ref{N-flow} and \ref{main result}) by using flows across cuts together with  ladders. 
The remainder is organized  as follows:  in the next section,  we give some  terminology,   useful lemmas and our main theorems.  In Section 3, we provide the proof of Theorem \ref{N-flow}. In Section 4, we provide the proof of Theorem \ref{main result} using homotopic classes and the intersection numbers of loops on the torus from algebraic topology. In the last section, as some consequences, for a toroidal polyhex we obtain a stronger conclusion for two nontrivial perfect matchings (i.e. with alternating hexagons) to lie in the same connected component of its resonance graph in terms of only flows across cuts (Theorem \ref{stronger}), and we give exactly numbers of trivial  perfect matchings (Remark \ref{PMnumber}). Finally, we also point out that this stronger criterion no longer holds for nanotubes.
\section{Preliminaries}

In this section  we start some definitions and useful results;   for undefined terminology, see \cite{Lovp86}. For a graph $G$, let $V(G)$ and $E(G)$ denote its vertex set and edge set, respectively. A graph is said to be \emph{bipartite graph} if its vertices are always colored white and black such that any pair of adjacent vertices receive different colors. 
We always assume that all bipartite graphs admit such a black-white coloring.

A \emph{matching} $M$ of a graph $G$ is a set of pairwise non-adjacent edges.
 A vertex  $v$ is said to be
\emph{covered} or \emph{saturated} by $M$ if some edge of $M$ is incident with $v$. A \emph{perfect matching} $M$ of $G$ is a matching   that covers every vertex of $G$.
A graph $G$ is \emph{matchable} if it has a perfect matching. For a matchable graph $G$, an edge of $G$ is said to be \emph{forbidden} if it does not lie in any perfect matching of $G$.
 A connected bipartite graph is called \emph{elementary} if it is matchable  and has no forbidden edges.

For a perfect matching $M$ of a bipartite graph $G$, a cycle $C$ (or path $P$) is called an \emph{$M$-alternating cycle} (or path) if the edges of $C$ (or $P$) appear alternately in $M$ and off $M$. A cycle of $G$ is \emph{nice} if $G-V(C)$ has a perfect matching. Equivalently, a cycle of a bipartite graph $G$ is nice if $G$ has a perfect matching $M$ such that $C$ is $M$-alternating.
For two edge-sets $A$ and $B$, the \emph{symmetric difference} of $A$ and $B$ is $A\triangle B=(A\cup B)\setminus(A\cap B)$. For an $M$-alternating cycle $C$ of $G$, $E(C)\triangle M$ is also a perfect matching of $G$.
For distinct perfect matchings $M$ and $M'$ of $G$, $M\triangle M'$ forms disjoint $(M,M')$-alternating cycles of $G$.
An $M$-alternating cycle $C$ of a plane bipartite graph $G$ is said to be \emph{proper (improper)} if every edge of $C$ belonging to $M$ goes from the white (black) end-vertex to the black (white) end-vertex by the clockwise orientation of $C$.

A \emph{nanotube} can be obtained as follows. In the hexagonal lattice, select a lattice point as the origin $O$ and let $\vec{a}_{1}$, $\vec{a}_{2}$ be the basic lattice vectors (see Fig. \ref{Figure2-2}(left)). Let $A$ be the lattice point determined by  $\overrightarrow{OA}=a \vec{a}_{1}+b\vec{a}_{2}$, where $a,b$ are integers and $(a,b)\neq(0,0)$. Let $l_1$ and $l_2$ be the lines through $O$ and $A$, respectively, perpendicular to $\overrightarrow{OA}$. Identify each point of $l_1$ and $l_2$ such that $A$ and $O$ are superimposed to obtain a hexagonal tessellation $H$ of the cylinder.
The axis of the cylinder is parallel to $l_1$, say $x$-axis.  A nanotube is the finite subgraph of $H$ lying between two vertex-disjoint cycles $C_1$ and $C_2$ (length at least is 4) of $H$ encircling the axis. This is called  an $(a,b)$-type nanotube and  cycles $C_1$ and $C_2$ are its two open-ends. A nanotube $N$ is a planar bipartite graph \cite{SHZ96,TZra15} and it can be drawn in the plane so that one end corresponds to a hole and the other end to the exterior face and all the other faces are hexagons. We always consider such a plane graph for a nanotube.
\begin{figure}
\centering
\includegraphics{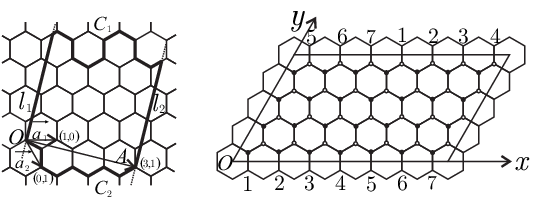}
\caption{\label{Figure2-2} A (3,1)-type nanotube and a toroidal polyhex $H(7,4,3)$.}
\end{figure}

Let $P$ be a $p \times q$-parallelogram ($p\geq2$ and $q\geq2$) section cut from the hexagonal lattice such that every corner lies on the center of a hexagon, two lateral sides pass through $q$ oblique edges, top and bottom sides pass through $p$ vertical edges (see Fig. \ref{Figure2-2}(right));
To construct  a \emph{toroidal polyhex} $H(p,q,t)$ (simply, say $H$), first identify two lateral sides of $P$ to form a cylinder, then bend the cylinder so that its top boundary is identified with the bottom boundary after a twist of $t$ hexagonal segments around the circumference. For convenience, we introduce an affine coordinate system $XOY$ for $H$ \cite{ZY08}: Take the bottom side and one lateral side of the $p \times q$-parallelogram $P$ as $x$-axis and $y$-axis such that two axes form an angle of $60^\circ$ and $P$ lies in non-negative region; The origin $O$ is the intersection of two axes. One unit of length is defined as the distance between a pair of parallel edges in a hexagon. Toroidal polyhex $H(p,q,t)$ is a bipartite graph \cite{SLZh05}.

Next, we provide the definition for the resonance graph via flips along hexagons. 

\begin{defi}
{\rm Let $G$ be a hexagonal system, coronoid system, nanotube, or  toroidal polyhex. The {\sl resonance graph}  $R_6(G)$ is a simple graph in which the vertices are the perfect matchings of $G$ and two perfect matchings $M_1$ and $M_2$ are joined by an edge if their symmetric difference consists of exactly six edges of a hexagon $s$ of $G$ (say $M_1$ and $M_2$ are transformed into each other by a flip along the hexagon $s$).}
\end{defi}

The resonance graph $R(G)$ of a connected plane bipartite graph $G$ is a graph whose vertices are the perfect matchings and two perfect matchings are adjacent  provided their symmetric difference forms a cycle that is  the boundary of an {\sl interior} face of $G$ (flip an interior face). We  present some known results for hexagonal systems and (plane)  bipartite graphs as follows.
\begin{lem}[\cite{ZGC88}]\label{hexagonal system}
The resonance graph of a matchable hexagonal system is a
connected bipartite graph.
\end{lem}
\begin{lem}[\cite{Lovp86}]\label{unique}
Let $G$ be a bipartite graph with a unique perfect matching. Then $G$ must contain at least one vertex of degree 1 in each color class.
\end{lem}
\begin{lem}[\cite{ZZ2000}]\label{forbidden edge}
Let $G$ be a matchable plane bipartite graph. Assume that a cycle $C = v_1v_2 \cdots v_{2n}v_1$ of $G$ lies in the boundary of a face of $G$. If an edge $v_{2i-1}v_{2i}$ ($1\leq i \leq n$) is a forbidden edge of $G$, then there exists $j$ such that $v_{2j}v_{2j+1}$ ($1\leq j \leq n$, where $v_{2n+1}:=v_{1}$) is also a forbidden edge of $G$.
\end{lem}

\begin{lem}[\cite{LZhan26}]\label{path}
Let $G$ be a plane elementary bipartite graph,  $C_0$ be the boundary of the exterior face of $G$ and $C_1,\ldots,C_n$  be the boundaries of interior faces of $G$.  Let $M_{1}$ be a  perfect matching of $G$ such that $C_0$ is a proper $M_{1}$-alternating cycle and $C_{1},\ldots,C_{n}$ are improper $M_{1}$-alternating cycles. Then $R(G)$ has a path between $M_{1}$ and $M_{2}:=M_1\triangle (\cup_{i=0}^n E(C_{i}))$ by flipping only those faces of $G$ in the region  obtained from the interior of  $C_{0}$ minus  $C_{i}$ ($1\leq i\leq n$) with their interiors.
\end{lem}

A cycle  of a graph embedded in surface (possibly with boundary) is called a \emph{contractible} if it can be contracted to a point; Otherwise, a   \emph{non-contractible cycle}. For a contractible cycle $C$ of a nanotube or a toroidal polyhex $H$, it bounds a hexagonal system, denoted by $H[C]$.
The following lemma is a simple but useful result.

\begin{lem}\label{contractible cycle}
Let $G$ be a nanotube or a toroidal polyhex and $M_{1}$, $M_{2}$ be two perfect matchings of $G$. If $M_{1} \triangle M_{2}$ forms a contractible cycle $C$, then $R_6(G)$ has a path between $M_{1}$ and $M_{2}$  by flipping only hexagons in $H[C]$.
\end{lem}

\begin{proof}
Since $C$ is a contractible cycle, $H[C]$ is a hexagonal system. Let $M_{1}^{*}:= M_1|_{H[C]}$ and $M_{2}^{*}:= M_2|_{H[C]}$.
Then $M_{1}^{*}$ and $M_{2}^{*}$ are two perfect matchings of $H[C]$ and $M_{1}^{*}\triangle  M_{2}^{*}=E(C)$.
By Theorem \ref{hexagonal system}, $R_{6}(H[C])$ has a path $M_{1}^{*}(=M_{1}')M_{2}'\cdots M_{t-1}'(M_{t}'=)M_{2}^{*}$. 
Let  $M_0:=M_{1}\setminus M_{1}^{*}=M_{2}\setminus M_{2}^{*}$. Then $R_6(G)$ has a path $M_{1}(M_{2}'\cup M_0)\cdots (M_{t-1}'\cup M_0)M_2$. That is, $R_{6}(G)$ has a path between $M_{1}$ and $M_{2}$  by flipping only hexagons  in $H[C]$.
\end{proof}

\subsection{Flow across a cut}
The \emph{dual graph} $G^{*}$ of a toroidal polyhex (or the plane drawing of a nanotube) $G$ is a graph that has one vertex at the center of each hexagon (or every face) of $G$, two vertices $f_{1}^{*}$ and $f_{2}^{*}$ are joined by an edge $e^{*}$ in $G^{*}$ if their corresponding faces $f_{1}$ and $f_{2}$ of $G$ have an edge $e$ in common ($e^{*}$ only crosses edge $e$ at the center and $e^*$ is straight in the interiors of hexagons).

\begin{defi}[\cite{LZhan26}]
{\rm A \emph{cut segment} $L$ is a path in the dual graph of a nanotube $N$ with a direction, and its end vertices correspond to two open-ends of $N$.
The set of edges of $N$ crossed by $L$ is called a \emph{cut}, denoted by $L^{*}$ (see Fig. \ref{Figure2-3}(a)). }
\end{defi}

\begin{figure}
\centering
\includegraphics[scale=0.8]{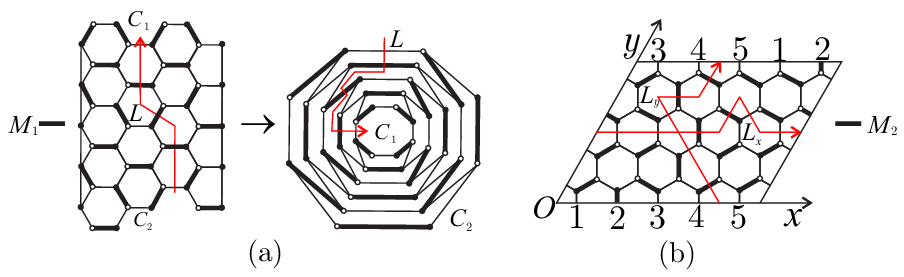}
\caption{\label{Figure2-3} (a) A $(4,2)$-type nanotube and its  planar drawing. (b) Toroidal polyhex $H(5,4,3)$.}
\end{figure}

Similarly, we define two cuts in a toroidal polyhex.
\begin{defi}
{\rm    An  oriented cycle $L_{x}$ (resp. $L_{y}$) of the dual graph of a toroidal polyhex $H$ is called a \emph{cut cycle} if it winds once around the $x$-axis (resp. $y$-axis) direction in the $p \times q$-parallelogram of $H(p,q,t)$. The set of edges of $H(p,q,t)$ crossed by $L_{x}$ (resp. $L_{y}$) is called a \emph{cut}, denoted by $L_{x}^*$ (resp. $L_{y}^{*}$) (see Fig. \ref{Figure2-3}(b)).  }
\end{defi}

Let $G$ be a nanotube or a toroidal polyhex, and $L$ be a cut segment or a cut cycle.

\begin{defi}[\cite{LZhan26}]
{\rm    The \emph{flow} of a given perfect matching $M$ of $G$ across $L$, denoted by ${\rm flow}_{L}(G, M)$, is the number of edges in $M$ crossed by $L$, where each such edge in $M$ is counted positively (resp. negatively) if its white end-vertex is to the left side (resp. right side) of $L$.}
\end{defi}

For convenience,
let $(L^{*})^{+}=\{e\in L^*|\mbox{the white end-vertex of $e$ is to the left side of  $L$}\}$
and $(L^{*})^{-}=\{e\in L^*|\mbox{the white end-vertex of $e$ is to the right side of $L$}\}$.
Then
\begin{equation}\label{dif}
    {\rm flow}_{L}(G, M)=|(L^{*})^{+}\cap M|-|(L^{*})^{-}\cap M|.
\end{equation}

Let $G_1$ be a subgraph of $G$ and ${\rm flow}_{L}(G_1, M):=|(L^{*})^{+}\cap M\cap E(G_1)|-|(L^{*})^{-}\cap M\cap E(G_1)|$.
Then we have
\begin{equation}
      {\rm flow}_{L}(G, M)={\rm flow}_{L}(G_1, M)+{\rm flow}_{L}(G-V(G_1), M).
\end{equation}

\begin{ex}
{\rm (i) Let $N$ be a nanotube with cut segment $L$ and  a perfect matching $M_{1}$ in Fig. \ref{Figure2-3}(a). Then ${\rm flow}_{L}(N, M_{1})=1$.

(ii) Let $H$ be toroidal polyhex $H(5,4,3)$ with   a perfect matching $M_{2}$ and oriented cut cycles $L_{x}$, $L_{y}$  in Fig. \ref{Figure2-3}(b). Then,   ${\rm flow}_{L_{x}}(H, M_{2})=0$
and ${\rm flow}_{L_{y}}(H, M_{2})=-1$.}
\end{ex}

Let $N$ be a matchable nanotube with a cut segment $L$, $M_{1}$, $M_{2}$ be two perfect matchings. We have the following two lemmas.

\begin{lem}[\cite{LZhan26}]\label{one II cycle}
If a contractible cycle $C$ of $N$ is $(M_{1},M_{2})$-alternating, then ${\rm flow}_{L}(C,M_{1})={\rm flow}_{L}(C,M_{2})$. Otherwise, $|{\rm flow}_{L}(C,M_{1})-{\rm flow}_{L}(C,M_{2})|=1$.
\end{lem}

\begin{lem}[\cite{LZhan26}]\label{two cycle}
 For two disjoint non-contractible  $(M_{1},M_{2})$-cycles $C_{1}$, $C_{2}$, if ${\rm flow}_{L}(C_{1}\cup C_{2},M_{1})={\rm flow}_{L}(C_{1}\cup C_{2},M_{2})$, then one of $C_{1}$ and $C_{2}$ is a proper $M_{1}$-alternating cycle and the other one is improper.
\end{lem}

\begin{lem}\label{one contractible cycle}
Let $H$ be a toroidal polyhex with two cut cycles $L_{x}$ and $L_{y}$, $M_{1}$ and $M_{2}$ two perfect matchings of $H$.  If $C$ is an $(M_{1},M_{2})$-alternating contractible cycle of $H$, then ${\rm flow}_{L_{j}}(C,M_{1})={\rm flow}_{L_{j}}(C,M_{2})$ ($j=x,y$).
\end{lem}
\begin{proof}
Since  $C$ is contractible cycle,  $|L_{j}^*\cap E(C)|$ ($j=x,y$) is even. By  similar arguments to Lemma \ref{one II cycle} in \cite{LZhan26}, the result holds.
\end{proof}

\begin{lem}\label{necessity}
Let $G$ be a nanotube or a toroidal polyhex and $L$ be the cut segment or cut cycle.   If two perfect matchings $M_{1}$ and $M_{2}$ of $G$
lie in the same connected component of  $R_6(G)$,
then  ${\rm flow}_{L}(G, M_{1})$=
${\rm flow}_{L}(G, M_{2})$.
\end{lem}

\begin{proof}
Suppose that $R_{6}(G)$ has a path $P = M_{1}'
(= M_1)M_{2}'\cdots M_{t-1}'M_{t}'(= M_{2})$
between $M_{1}$ and $M_{2}$, where $M_{i+1}^{'}=M_{i}^{'}\triangle E(S_{i})$
and $S_{i}$ is an $(M_{i}^{'},M_{i+1}^{'})$-alternating hexagon
of $G$, $i=1,2,\ldots,t-1$.
For each $i$,  the boundary of $S_{i}$ is
a contractible cycle of $G$, by Lemmas \ref{one II cycle} and \ref{one contractible cycle}, we have ${\rm flow}_{L}(S_i, M_{i}') =
{\rm flow}_{L}(S_i, M_{i+1}')$.
Thus ${\rm flow}_{L}(G, M_i')= {\rm flow}_{L}(G, M_{i+1}')$ as $M_{i+1}^{'}\triangle M_{i}^{'}= E(S_{i})$,  which implies that ${\rm flow}_{L}(G, M_1) = {\rm flow}_{L}(G, M_2)$.
\end{proof}

\subsection{Ladders of a nanotube and toroidal polyhex}

For a hexagon of a nanotube or toroidal polyhex, the edges have three \emph{edge-directions}, denoted by $a$, $b$, and $c$ (see Fig.  \ref{Figure2-4}). A cycle $C$ of a nanotube or toroidal polyhex is called a \emph{zigzag cycle} if the edges of $C$ have only two edge-directions and appear alternately in the two directions.
Note that if a nanotube $N$ contains an edge parallel to $l_{1}$, then $N$ admits a zigzag cycle; A toroidal polyhex admits zigzag cycles of three types:   $a, c$-edge directions cycles (I-direction cycles \cite{WYZ08}),  $b, c$-edge directions cycle (II-direction cycles), and $a, b$-edge directions (III-direction cycles), for an example see Fig. \ref{Figure2-4}.

\begin{rem}\label{zigzag}
{\rm We can see that the toroidal polyhex $H(p,q,t)$ has $q$ III-direction cycles.  By Lemma 2.1 in \cite{WYZ08},   $H(p,q,t)$ has $\frac{p}{g}$ I-direction cycles (resp. $\frac{p}{g'}$ II-direction  cycles) where $g$ (resp. $g'$) is the smallest positive integer satisfying $gt\equiv0 (\text{mod } p)$ $($resp. $g'(q + t)\equiv0 (\text{mod } p)$$)$.}
\end{rem}

For a perfect matching $M$ of a nanotube or toroidal polyhex, a \emph{ladder} $Z$ of $M$ is an $M$-alternating zigzag cycle. 
For  $\emptyset\not=S\subset V(G)$, let $\partial(S)$ denote the set of edges with precisely one end in $S$.
Clearly, $\partial(V(Z))$ has no edges in $M$ (see Fig. \ref{Figure2-4}).
\begin{figure}
\centering
\includegraphics[scale=0.8]{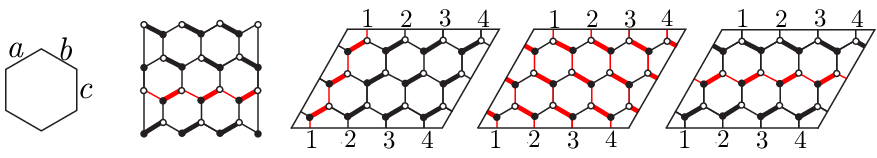}
\caption{\label{Figure2-4}The red cycles are ladders of a perfect matching $M$ (thick edges) in a nanotube or toroidal polyhex.}
\end{figure}
\begin{lem}\label{same ladder}
Let $G$ be a nanotube or toroidal polyhex. If two perfect matchings of $G$  lie in the same connected component of the resonance graph of $G$, then they have precisely the same ladders.
\end{lem}

\begin{proof}
Let $M_{1}$, $M_{2}$ be two perfect matchings of $G$. Suppose $R_{6}(G)$ has a path $P = M_{1}'(= M_1)M_{2}'\cdots M_{t-1}'M_{t}'(= M_{2})$ between $M_{1}$ and $M_{2}$.
For any consecutive two perfect matchings $M_i'$ and $M_{i+1}'$, $M_i\triangle M_{i+1}'$ is the set of  six edges of some hexagon $S_{i}$ in $G$.
Let $E(S_{i})=\{e_{1},e_{2},e_{3},e_{4},e_{5},e_{6}\}$ (see Fig. \ref{Figure2-5}). \begin{figure}[H]
\centering
\includegraphics[scale=0.9]{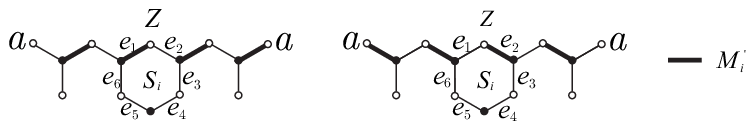}
\caption{\label{Figure2-5}Illustration for  the proof of Lemma \ref{same ladder}.}
\end{figure}
If $M_{i}'$ has a ladder $Z$, then
we claim that $Z$ and $S_{i}$ are disjoint.
Otherwise,  $Z$ contains a pair of adjacent edges of $S_{i}$, say $e_{1},e_{2}$.
If $e_{1}\in M_{i}'$ (see the left graph in Fig. \ref{Figure2-5}), then $e_{3} \in M_{i}'$ since $S_{i}$ is an $(M_{i}',M_{i+1}')$-alternating cycle. But $e_{3}\in \partial(V(Z))$, so $e_{3} \notin M_{i}'$, a  contradiction.
If $e_{1}\notin M_{i}'$ (see the right graph in Fig. \ref{Figure2-5}), then $\{e_{2},e_{4},e_{6}\}\subseteq M_{i}'$. But $e_{6}\in \partial(V(Z))$, so $e_{6} \notin M_{i}'$, a  contradiction.  Hence the claim holds. So  $Z$ belongs to $E(G)\setminus E(S_{i})$. Since $M_{i}'$ and $M_{i+1}'$ are the same restrictions on  $E(G)\setminus E(S_{i})$,  $M_{i+1}'$ has also the ladder $Z$. Similarly, if $M'_{i+1}$ has a ladder $Z$, then $M_{i}'$ has the ladder $Z$.
\end{proof}

\subsection{Main results}

In this paper we  mainly give simple criteria  for two perfect matchings of a nanotube and a toroidal polyhex to belong to the same connected component of their resonance graphs using the flow and ladder.
\begin{thm}\label{N-flow}
Let $N$ be a matchable nanotube with a cut segment $L$. Two perfect matchings of $N$ lie in the same connected component of the resonance graph  $R_6(N)$ if and only if  they have the same  flows across $L$  and  precisely the same ladders.
\end{thm}


\begin{thm}\label{main result}
Let $H$ be a toroidal polyhex with two cut cycles $L_{x}$ and $L_{y}$. Two perfect matchings of $H$ lie in the same connected component of the resonance graph  $R_{6}(H)$ if and only if they have the same  flows across the cut cycles  and precisely the same ladders.
\end{thm}

\section{Proof of Theorem \ref{N-flow}}
In order to show Theorem \ref{N-flow}, we first provide some useful lemmas. Let $s$ be a hexagon of a nanotube $N$ and $e$ be an edge of $s$.  Let $L_{s,e}$ be an oriented  perpendicular
bisector of $e$ such that it starts at the midpoint of $e$, passes first  through $s$, and ends at $e$ or the boundary of $N$ (see Fig. \ref{Figure2-7}).

\begin{lem}\label{different edge-directions}
Let $N$ be a non-elementary matchable nanotube. Then $N$ admits  a forbidden edge $e$  and a perfect matching $M$ so that the two edges $e_{1}$ and $e_{2}$ of $M$ covering the end vertices of $e$ have different edge-directions; Further, let $s_{1}$ be the hexagon of $N$ containing $e$, $e_1$ and $e_2$. Then the edges of $N$ which intersect $L_{s_{1},e}$ are forbidden edges.
\end{lem}
\begin{proof}
Since $N$ is a bipartite graph with a perfect matching,  $N$ has at least two perfect matchings. Otherwise, by Lemma \ref{unique}, $N$ has a 1-degree vertex, a contradiction.
Thus, $N$ has a nice cycle. Let  $N'$ be  the union of all nice cycles of $N$.  Since $N$ is non-elementary,  $N$ has a forbidden edge, not in $N'$.
So $N' \subset N$. Since $N$ is connected, we can choose an edge $e$ of $N$ not in  $N'$ such that $e$ is adjacent to one edge, say $e_{1}$, of $N'$. Then $e$ is a forbidden edge and  $e_{1}$ is in an $M$-alternating cycle $C$ for a perfect matching $M$  of $N$.
Let $x$ be the common end-vertex of $e$ and $e_{1}$ and let $y$ be the other end-vertex of $e$. Let $e_{2}$ be the edge of $M$ which covers $y$. We claim that $e_{2}$ is not an edge of $C$. Otherwise,  $e$ joins two distinct colored vertices of $C$  since $N$ is a bipartite graph. Hence $e$ and the part of  $C$ form an $M$-alternating cycle, a contradiction.
If $e_{1}$ and $e_{2}$ have different edge-directions, then we are done. Otherwise, for another perfect matching $M \triangle E(C)$   of $N$, the result holds.

\begin{figure}[ht]
\centering
\includegraphics[scale=0.8]{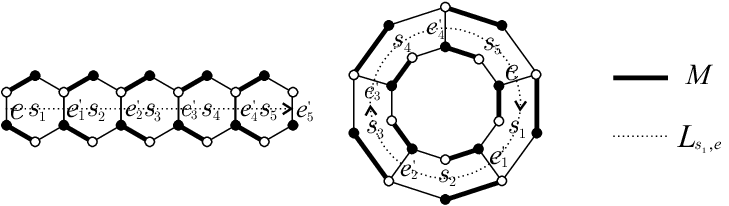}
\caption{\label{Figure2-6} Illustration for the proof of Lemma \ref{different edge-directions}.}
\end{figure}
Next, we show that the edges of $N$ which intersect $L_{s_{1},e}$ are forbidden edges. Let $e, e_1', e_2', \ldots, e_n'$ be the edges of $N$ which intersect $L_{s_{1},e}$ in turn. Then  either $e_n'$ is a boundary edge of $N$ or $e_n'=e$ (see Fig. \ref{Figure2-6}). Let   $s_1, s_2, \ldots, s_n$  be a sequence of hexagons such that   $s_i$ and $s_{i+1}$ have a common edge $e_i'$  for $i = 1, \ldots, n-1$. Since $e$ is a forbidden edge of $N$ and $e_1,e_2\in M$, by Lemma \ref{forbidden edge} to hexagon $s_1$ we have that $e'_1$ is a forbidden edge of $N$. Hence the two edges of $s_2$ that adjacent to $e_1'$ must belong to $M$.  By Lemma \ref{forbidden edge} to hexagon $s_2$ we have that $e'_2$ is a forbidden edge of $N$.
Repeating the above process we arrive at that  the two edges of  each $s_i$  ($2\leq i\leq n$) that adjacent to $e_{i-1}'$ must belong to $M$ and each $e'_i$ is a forbidden edge of $N$, for $1\leq i\leq n$.
\end{proof}

\begin{lem}\label{elementary nanotube}
Let $N$ be a nanotube, $M_{1}$ and $M_{2}$ two perfect matchings of $N$ such that the open-ends $C_{1}$ and $C_{2}$ are $(M_{1},M_{2})$-alternating cycles.
If $M_{1}$ and $M_{2}$  have precisely the same ladders,  then $N$ is an elementary nanotube.
\end{lem}
\begin{proof}
Suppose, to the contrary, that $N$ is not elementary. Then, by Lemma \ref{different edge-directions},  $N$ has a forbidden edge $e$ and a perfect matching $M$  so that two edges $e_{1},e_{2}\in M$ which cover the end vertices of $e$ have different edge-directions. Since $e$ does not lie in the boundary, $N$ admits a hexagon $s$ containing $e,e_{1},e_{2}$. Let $E'$ be the set of edges of $N$ intersected by $L_{s,e}$ and $e_{n}$ the last an edge. Then the edges of $E'$ have the same edge-direction and $e,e_n\in E'$.
By Lemma \ref{different edge-directions}, we have that all the edges of $E'$  are forbidden edges.  If $e$ is not parallel to $l_{1}$, then $e_{n}$ is in $C_{1}$ or $C_{2}$ (see Fig. \ref{Figure2-7}(b-c)),
a contradiction. Hence $e$ is parallel to $l_{1}$ and $e=e_n$, say vertical edge  (see Fig. \ref{Figure2-7}(a)).
\begin{figure}[ht]
\centering
\includegraphics[scale=0.7]{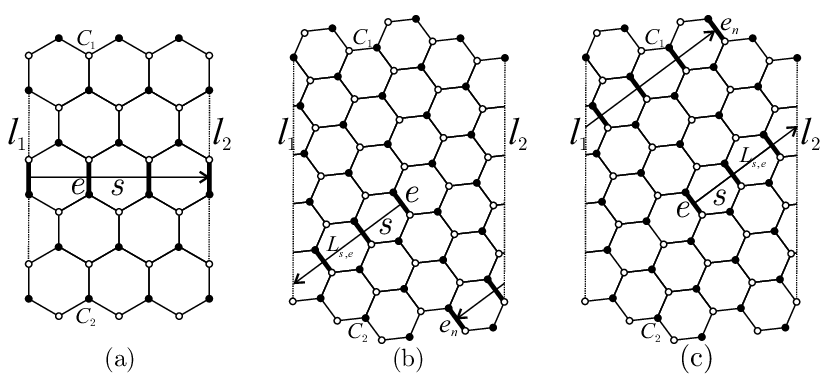}
\caption{\label{Figure2-7} Illustration for the proof  of Lemma \ref{elementary nanotube}.}
\end{figure}
Let $C$ be a zigzag cycle whose edges are oblique and incident with the white  end-vertices of all edges of $E'$. Since the edges of $E'$ are forbidden, $C$ is an $M'$-alternating cycle for any perfect matching $M'$ of $N'$. Hence the vertical edges incident with vertices of $C$ are forbidden edges, which are thus non-boundary edges.
Repeating this process, finally we arrive in that all vertical edges of $N'$ are forbidden and non-boundary edges, and all oblique edges form disjoint $M'$-alternating zigzag cycles.
Thus $C_{1}$ and $C_{2}$ are $M'$-alternating zigzag cycles.
Since $C_{1}$ and $C_{2}$ are $(M_{1},M_{2})$-alternating cycles,
$M_{1}$ and $M_{2}$ have different ladders in $C_{1}$ and $C_{2}$,
 contradicting  that $M_{1}$ and $M_{2}$ have same ladders.
\end{proof}

\noindent{\bf Proof of Theorem \ref{N-flow}.} Let $M_{1}$ and $M_{2}$ be two perfect matchings of a nanotube $N$.  If $M_{1}$ and $M_{2}$  lie in the same connected component of $R_6(N)$, then  ${\rm flow}_{L}(N,M_{1})={\rm flow}_{L}(N,M_{2})$ by Lemma \ref{necessity}, and $M_{1}$ and $M_{2}$ have precisely the same ladders by Lemma \ref{same ladder}.   So the necessity  holds.

Conversely, suppose that ${\rm flow}_{L}(N,M_{1})= {\rm flow}_{L}(N,M_{2})$ for cut segment $L$ and $M_{1}$, $M_{2}$ have precisely the same ladders. We shall show that $M_{1}$ and $M_{2}$ lie in the same connected component of $R_6(N)$. It suffices to show that $R_{6}(N)$ has a path between $M_{1}$ and $M_{2}$.
Let $\mathcal{C}=M_{1}\triangle M_{2}=\cup_{i=1}^k C_{i}$, $k\geq0$, where the $C_{i}$'s are disjoint $(M_1,M_2)$-alternating cycles. We proceed by induction on the number $k$.
If $k=0$, it is trivial.
Next, let $k\geq 1$. If $\mathcal{C}$ contains a contractible cycle $C$, then $H[C]$ is a hexagonal system. Let $M_{3}=M_{1}\triangle E(C)$. Then $M_{3}$ is also a perfect matching of $N$. Since $M_{1}\triangle M_{3}= E(C)$, by Lemma \ref{contractible cycle}, $R_{6}(N)$ has a path between $M_{1}$ and $M_{3}$.
By Lemma \ref{necessity}, ${\rm flow}_{L}(N,M_{3})= {\rm flow}_{L}(N,M_{1})={\rm flow}_{L}(N,M_{2})$. By Lemma \ref{same ladder}, $M_{1},M_{2},M_{3}$ have precisely the same ladders.
Since $M_{3}\triangle M_{2}=\mathcal{C}\setminus\{C\}$, by the induction hypothesis, $R_{6}(N)$ has a path
between $M_{3}$ and $M_{2}$. Thus $R_{6}(N)$ has a path between $M_1$ and $M_{2}$.

So suppose that $\mathcal C$ has no contractible cycles.
Since ${\rm flow}_{L}(N,M_{1})={\rm flow}_{L}(N,M_{2})$, ${\rm flow}_{L}(\mathcal{C}, M_{1})={\rm flow}_{L}(\mathcal{C},M_{2})$ as $M_{1}\triangle M_{2}=\mathcal{C}$, which implies the following.

\begin{equation}\label{3.5}
    \sum\limits_{i=1}^{k}({\rm flow}_{L}(C_i,M_{1})-{\rm flow}_{L}(C_i,M_{2}))=0.
\end{equation}

Since each $C_{i}$ is non-contractible cycle,
$|{\rm flow}_{L}(C_{i},M_{1})-{\rm flow}_{L}(C_{i},M_{2})|=1$ by Lemma \ref{one II cycle},
for $1\leq i\leq k$. So Eq. (\ref{3.5}) implies that $k$ is even:
  half of the cycles  $C_{1},\ldots,C_{k}$ yield
${\rm flow}_{L}(C_{i},M_{1})-{\rm flow}_{L}(C_{i},M_{2})=1$, while
the other half  yield $-1$.
By Lemma \ref{two cycle},  half of the cycles in $\mathcal C$
are proper $M_{1}$-alternating cycles and
the other half  are improper $M_{1}$-alternating
cycles when drawing $N$ in the plane. So we can take a pair of proper and improper $M_1$-alternating cycles $C_1$ and $C_2$ in $\mathcal C$.  We use
$N'=N[C_{1},C_{2}]$ to denote sub-nanotube with
two open-ends $C_1$ and $C_2$ from $N$. Then the restriction $M'_i$ of $M_i$ on $N'$ is a perfect matching of $N'$, $i=1,2$, and $M_{1}'$, $M_{2}'$ have precisely the same ladders since $M_{1}$ and $M_{2}$ share the same ladder structure. By Lemma \ref{elementary nanotube}, $N'$ is an elementary nanotube.

Let $M_3=M_1\triangle E(C_1)\triangle E(C_2)$ and
 $M_{3}^{'}=M_{3}|_{N'}$. Then $M_{1}^{'}\triangle M_{3}^{'}=E(C_1)\cup (C_2)$.
Since $N'$ is regarded as  a plane elementary bipartite graph so that  $I[C_{2}]\subset I[C_{1}]$. By Lemma \ref{path},  $R_{6}(N')$ has a path between  $M_{1}'$ and  $M_{3}'$. So $R_{6}(N)$ has a path between  $M_{1}$ and  $M_{3}$.  From the necessary,
${\rm flow}_{L}(N,M_{3})= {\rm flow}_{L}(N,M_{1})={\rm flow}_{L}(N,M_{2})$
and $M_{1},M_{2},M_{3}$ have precisely the same ladders. Since $M_3\triangle M_2=\mathcal C\setminus \{C_1,C_2\}$,
by the induction hypothesis, $R_{6}(N)$ has a path
between $M_{3}$ and $M_{2}$. Thus $R_{6}(N)$ has a path between $M_1$ and
$M_{2}$.
\hfill $\square$
\section{Proof of Theorem \ref{main result}}

In algebraic topology, a loop of a torus $T$ is a continuous map from closed interval $[0,1]$ to $T$ with a starting and ending at the same point. The fundamental group  of a topological space is the set of homotopy classes of loops based at a given point, equipped with the operation of path composition.
The fundamental group of a torus $T$ is isomorphic to $\mathbb{Z}\times \mathbb{Z}$ \cite{Hatc01}.
Under this isomorphism, a pair $(m,n) \in \mathbb{Z} \times \mathbb{Z}$ corresponds to a loop of $T$, denoted by $C_{m,n}$, which winds $m$ times in the  longitude (through the hole, i.e. along $x$-axis)  while winding $n$ times in the latitude around the torus $T$ (along $y$-axis). Note that $(m,n)$ is the \emph{winding vector} of $C_{m,n}$ \cite{Gree09}.
One could also allow negative values for $m$ or $n$, corresponding to traversing the loop in the opposite direction.  $(m,n)$ and $(-m,-n)$ may actually represent a pair of loops of opposite orientations.

An oriented cycle  of  a toroidal polyhex $H$ is a simple loop. A cycle  $C$ of $H$ is said to have a winding vector $(m,n)$ if $(m,n)$ is the winding vector of some orientation of $C$.   A $C_{m,n}$-cycle is an oriented cycle with the winding vector $(m,n)$.

Two loops of a torus are said to be {\it freely homotopic} if one can be continuously deformed into the other. A loop  freely homotopic to $C_{0,0}$ is a contractible loop.
Every non-contractible loop in the torus is freely homotopic to $C_{m,n}$ for some non-zero $(m, n) \in \mathbb{Z} \times \mathbb{Z}$.
It is well known that $C_{m,n}$ is represented by a simple loop if and only if  ${\rm gcd}(m,n)=1$ (the greatest common divisor of $m$ and $n$) (see \cite[Theorem 1]{Scha76} ).

\begin{defi}\label{geometric intersection}
{\rm \cite{Schr93} The geometric intersection number of two homotopy classes of loops $C_{m,n}$ and $C_{m',n'}$ is $\operatorname{mincr}(C_{m,n}, C_{m',n'}) = \min\{|C\cap D| \}$ where $C$ and $D$ range over all loops freely homotopic to $C_{m,n}$ and $C_{m',n'}$, respectively.}
\end{defi}
 Moreover,
\begin{equation}\label{number of crossing}
\operatorname{mincr}(C_{m,n}, C_{m',n'}) = |mn' - m'n|,
\end{equation}
for all $m, n, m', n' \in \mathbb{Z}$ \cite{Schr93}.

\begin{thm}\label{two types}
Let $\mathcal C$ be a set of disjoint non-contractible cycles in a toroidal polyhex $H$ with $|\mathcal C|\geq2$.  Then   all cycles of  $\mathcal C$  can have the same winding vector.
\end{thm}
\begin{proof}
We can classify the non-contractible cycles of $H$ into the following three types:
 $C_{0,n}$-cycles ($n\neq0$),  $C_{m,0}$-cycles ($m\neq0$), and   $C_{m,n}$-cycles ($mn\neq 0$). We claim that all cycles of $\mathcal C$ have the same  type. Otherwise, suppose that two cycles $C_{1}$ and $C_{2}$ of $\mathcal C$ are of different types.
If $C_{1}$ is
a $C_{0,n}$-cycle and $C_{2}$ is a $C_{m,0}$-cycle, by   Eq. (\ref{number of crossing}) we have $|C_{1}\cap C_{2}|\geq \operatorname{mincr}(C_{0,n}, C_{m,0}) = |0\times 0 - m\times n|=|mn|\geq 1$, which implies that $C_{1}$ and $C_{2}$ have a vertex in common, a contradiction. The proofs for the other cases  are similar.

We now show that all cycles of $\mathcal C$ can have the same winding vector.
If the cycles of $\mathcal C$ are $C_{0,n}$-cycles ($n\neq0$), then $|n|={\rm gcd}(0,n)=1$ since they are simple loops.
If the cycles of $\mathcal C$ are  $C_{m,0}$-cycles ($m\neq0$), similarly we have $|m|=1$.  Next,  assume that all cycles of $\mathcal C$ are $C_{m,n}$-cycles ($mn\neq 0$). Let two cycles  $C_{1}$ and $C_{2}$  of $\mathcal C$ be  $C_{m,n}$-cycle and $C_{m',n'}$-cycle respectively.
By Eq. \ref{number of crossing}, $\operatorname{mincr}(C_{m,n}, C_{m',n'})=|mn' - m'n|\leq |C_{1}\cap C_{2}|$.
Since $C_{1}$ and $C_{2}$ are disjoint cycles, we have $|mn' - m'n|=0$. So  $mn'=m'n$. Since $mn\neq0$ and $m'n'\neq0$, let $t=\frac{m}{m'}=\frac{n}{n'}$. Then $m = tm'$, $n = tn'$.
We claim that $|t|=1$. Let $t=\frac{a}{b}$, where $a\neq 0$ and $b> 0$ are a pair of coprime integers. So  $mb=am'$ and $nb=an'$, which imply that $b\mid m'$ and $b\mid n'$
as ${\gcd}(a,b)=1$.   As ${\rm gcd}(m',n')=1$,  $b=1$. Thus, $t$ is an integer.
Since $1={\rm gcd}(m,n)={\rm gcd}(tm',tn')=|t|{\rm gcd}(m',n')=|t|\times 1=1$, $|t|=1$, and $(m,n)=\pm (m',n')$. 
\end{proof}
\begin{lem}[\cite{Levi63}]\label{cylinder}
Let $C_{1}$ and $C_{2}$ be disjoint simple loops on a closed surface $V$.  If $C_{1}$ and $C_{2}$ are freely homotopic but not homotopic to zero, then $C_{1}$ and $C_{2}$ bound a cylinder in $V$.
\end{lem}

We can compute the winding vector of an oriented cycle $C$ in a toroidal polyhex $H(p,q,t)$  (see \cite{EL21}).  We re-choose the $y$-axis  to obtain a new parallelogram $P'$ so that the corresponding points on the top and bottom sides are identified  (see Fig. \ref{Figure2-8}). In this way we also obtain $H(p,q,t)$.  For $P'$, crossing the vertical side once rightward (resp. leftward)  contributes $1$ (resp. $-1$) to the first  coordinate $m$ of the winding vector, and crossing the horizontal side one upward (resp. downward)  contributes $1$ (resp. $-1$) to the second coordinate $n$.  

From a geometric perspective, a non-contractible $C_{m,n}$-cycle $C$ for $n\not= 0$ can be decomposed  into segments whose endpoints lie on the top and bottom sides of $P$. These segments are of two types: those with both endpoints on the same side of $P$, and those with endpoints on different sides of $P$.  The  latter are
{\it  columns} of $C$. Since there may exist two columns of $C$ upward and downward when going along $C$ on $P$; for an example,  see columns $I_{3}$ from 1 to 6 and $I_{4}$  from 7 to 3 in Fig. \ref{Figure2-8}.
\begin{figure}\centering\includegraphics[scale=0.7]{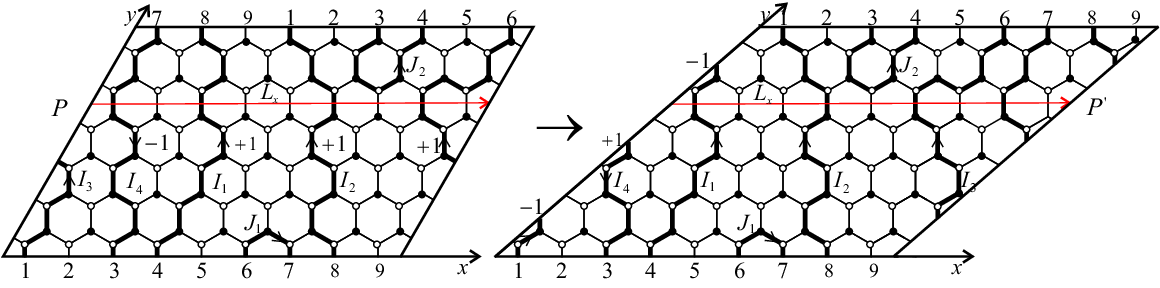}\caption{\label{Figure2-8} Two constructions of $H(9,6,3)$ with  a $(-1,2)$-cycle $C$ (thick line).}
\end{figure}
Thus, cycle $C$ has at least $n$ columns. Such segments of $C$ are always oriented along $C$.

Let $L_{x}$ and $L_{y}$ be any chosen two cut cycles of a toroidal polyhex $H$.  To avoid $L_x$ crossing the same edge of two segments of a cycle $C$, we choose $L_x$ so that it contains no points on the $x$-axis. 
For example, see Fig. \ref{Figure2-8}.

The following lemma gives the flow variation  of  perfect matchings $M_{1}$, $M_{2}$ across a  cut cycle on an $(M_{1},M_{2})$-alternating cycle.

\begin{lem}\label{one cycle}
Let $M_{1}$ and $M_{2}$ be two perfect matchings of a toroidal polyhex $H$ and $C$ an $(M_{1},M_{2})$-alternating cycle.

{\rm(i)}  If $C$ has the winding vector $\pm(1,0)$, then
\begin{equation}\label{C_{2}}
|{\rm flow}_{L_{y}}(C,M_{1})-{\rm flow}_{L_{y}}(C,M_{2})|=1.
\end{equation}

{\rm(ii)} If $C$ has  a winding vector $(m,n)$ for  $n\geq1$, then
\begin{equation}\label{C_{1}}
|{\rm flow}_{L_{x}}(C,M_{1})-{\rm flow}_{L_{x}}(C,M_{2})|=n.
\end{equation}
\end{lem}

\begin{proof}
(i) Orient $C$ to obtain a $C_{1,0}$-cycle (resp. $C_{-1,0}$-cycle). Then  $C$ traverses $p\times q$-parallelogram $P$ once from the left (resp. right) side to the right (resp. left) side. Then  $|L_{y}^*\cap E(C)|$ is odd. Let $L_{y}^*\cap E(C)=\{e_{1},e_{2},\ldots,e_{2r},e_{2r+1}\}$, where are sequentially labeled along the orientation of $C$.
For  a pair of successive edges $e_{i},e_{i+1}$, if the white end-vertices of $e_{i},e_{i+1}$ are on the same side of $L_{y}$, say   the left of $L_{y}$, then $e_{i},e_{i+1}\in (L^{*})^{+}$ and $(L^{*})^{-}\cap \{e_{i},e_{i+1}\}=\emptyset$. Since a path in $C$ between the same colored vertices  is of even length, one of $e_{i}$ and $e_{i+1}$ is in $M_{1}$, and the other one must be in $M_{2}$. Then $|(L_{y}^{*})^{+} \cap M_{1}\cap \{e_{i},e_{i+1}\}|-|(L_{y}^{*})^{-} \cap M_{1}\cap \{e_{i},e_{i+1}\}|=|(L_{y}^{*})^{+} \cap M_{2}\cap \{e_{i},e_{i+1}\}|-|(L_{y}^{*})^{-} \cap M_{2}\cap \{e_{i},e_{i+1}\}|=1-0=1$. If the white end-vertices of $e_{i},e_{i+1}$ are on different sides of $L_{y}$, say the white end-vertex of $e_{i}$ (resp. $e_{i+1}$) is on the left (resp. right) of $L_{y}$, then $e_{i}\in (L_{y}^{*})^{+}$ and $e_{i+1}\in (L_{y}^{*})^{-}$. In this case  $\{e_i,e_{i+1}\}\subseteq M_1$  or  $M_2$.  So  $|(L_{y}^{*})^{+} \cap M_{j}\cap \{e_{i},e_{i+1}\}|-|(L_{y}^{*})^{-} \cap M_{j}\cap \{e_{i},e_{i+1}\}|=0$, $j=1,2$.
In short, $M_{1}$ and $M_{2}$ have the same  flow across $L_{y}$ on the first $r$ pairs of edges.
Thus, ${\rm flow}_{L_{y}}(C,M_{1})-{\rm flow}_{L_{y}}(C,M_{2})$
is determined by the last edge $e_{2r+1}$. Since $e_{2r+1}$ belongs to either $M_1$ or $M_2$,  $|{\rm flow}_{L_{y}}(C,M_{1})-{\rm flow}_{L_{y}}(C,M_{2})|=1$.

(ii) Orient $C$ to obtain a $C_{m,n}$-cycle ($n\geq 1$). Then there exists $n+2s$ ($s\geq 0$) columns $I_{1},\ldots, I_{n+2s}$ and $t$ ($t\geq 0$) segments $J_{1},\ldots, J_{t}$
such that $$C=I_{1}\cup \cdots \cup I_{n}\cup I_{n+1}\cup \cdots \cup I_{n+2s}\cup J_{1}\cup \cdots \cup J_{t},$$ where the first $n$ columns have the same direction on  $P$;  half of the remaining columns are from bottom to top, and the other half from top to bottom on the  $P$  and  the endpoint of $J_{w}$ ($1\leq w\leq t$) belong to same side on $P$.  Along the orientation of $C$, if an edge $e$ has a  half in a segment $A_{i}$ and next half in a segment $A_{i+1}$, then  $A_{i}$ and $A_{i+1}$ are successive segments  and set $e$ belong to $ A_{i}$.  For  $j=1,2$, we have  \begin{equation}\label{4.5}
{\rm flow}_{L_{x}}(C,M_{j})=
\sum\limits_{i=1}^{n} {\rm flow}_{L_{x}}(I_{i},M_{j})+\sum\limits_{i=n+1}^{n+2s} {\rm flow}_{L_{x}}(I_{i},M_{j})+\sum\limits_{w=1}^{t} {\rm flow}_{L_{x}}(J_{w},M_{j}).
\end{equation}

\noindent {\bf{ Claim 1.}} If  the segment $J_{1}$ has  two endpoints in the same side on $P$, then \begin{equation}\label{J_{1}}
{\rm flow}_{L_{x}}(J_{1},M_{1})-{\rm flow}_{L_{x}}(J_{1},M_{2})=0.
\end{equation}
\begin{proof}
Since two endpoints of $J_{1}$  are in the same side on $P$, both are in the same side on $L_x$, so $|L_{x}^*\cap E(J_{1})|$ is even. By the same argument as in (i), $M_{1}$ and $M_{2}$ have the same  flows across $L_{x}$ on a  pair
of successive edges in $L_{x}^*\cap E(J_{1})$. Thus, Eq. (\ref{J_{1}}) holds. \end{proof}
\noindent {\bf{ Claim 2.}}
 If  two columns $I_{1}$ and $I_{2}$   have opposite orientations  on $P$, then \begin{equation}\label{C_{3}}
{\rm flow}_{L_{x}}(I_{1}\cup I_{2},M_{1})-{\rm flow}_{L_{x}}(I_{1}\cup I_{2},M_{2})=0.
\end{equation}
\begin{proof}
Without loss of generality, suppose that $I_{1}$  is along  $y$-axis. Then $I_2$ is in the direction opposite to the $y$-axis.  Since the endpoints of $I_{j}$ lie on the two sides of $L_{x}$, $j=1,2$, $|L_{x}^*\cap E(I_{j})|$  is odd. Similar to (i), we know that ${\rm flow}_{L_{x}}(I_{j},M_{1})-{\rm flow}_{L_{x}}(I_{j},M_{2})$
is determined by the last edge in $L_{x}^*\cap E(I_{j})$, say $e_{j}$.
 If the white end-vertices of $e_{1},e_{2}$ are on the same side of $L_{x}$, say left side, then $e_{1}$ is from black end-vertex to white end-vertex and $e_{2}$ is from white end-vertex to black end-vertex along $C$. Since $C$ is an $(M_1,M_2)$-alternating cycle,  one of $e_{1}$ and $e_{2}$ is in $M_{1}$  and the other one is in $M_{2}$.
Without loss of generality, suppose $e_{1}\in M_{1}$ and $e_{2}\in M_{2}$. Then
${\rm flow}_{L_{x}}(I_{1},M_{1})-{\rm flow}_{L_{x}}(I_{1},M_{2})=1$
and ${\rm flow}_{L_{x}}(I_{2},M_{1})-{\rm flow}_{L_{x}}(I_{2},M_{2})=-1$. Summing the above equalities to get Eq. (\ref{C_{3}}).
If the white end-vertices of $e_{1},e_{2}$ are on the different sides of $L_{x}$, without loss of generality  assume that the white end-vertex of $e_{1}$ (resp. $e_{2}$) is on the left (resp. right) side of $L_{x}$. Then  $\{e_1, e_{2}\}\subseteq M_1$  or  $M_2$. For the former,   ${\rm flow}_{L_{x}}(I_{1}\cup I_2,M_{1})-{\rm flow}_{L_{x}}(I_{1}\cup I_2 ,M_{2})=1-1=0$.
For the latter,  Eq. (\ref{C_{3}}) holds similarly. 
\end{proof}

By Claims 1 and  2 and Eq. (\ref{4.5}), we have
\begin{equation}\label{redu}
|{\rm flow}_{L_{x}}(C,M_{1})-
{\rm flow}_{L_{x}}(C,M_{2})|=
|\sum\limits_{i=1}^{n}({\rm flow}_{L_{x}}(I_{i},M_{1})-
{\rm flow}_{L_{x}}(I_{i},M_{2}))|.\end{equation}
Since two endpoints of each $I_{i}$ lie on the different sides on $L_{x}$, $|L_{x}^*\cap E(I_{i})|$ is odd.
Similar to (i), we know that
${\rm flow}_{L_{x}}(I_{i},M_{1})-{\rm flow}_{L_{x}}(I_{i},M_{2})$
is determined by the last edge  in $L_{x}^*\cap E(I_{i})$, say $e_i$. Since $I_{1},\ldots,I_{n}$ have the same direction on $P$, all the $e_i$ along $C$, $0\leq i\leq n$, cross $L_x$ from the same side.
Without loss of generality,
suppose that the white end-vertices of $e_{1},e_{2},\ldots,e_{n_{1}}$ $(0\leq n_{1}\leq n$)
are on the left of $L_{x}$ and the white end-vertices of
$e_{n_1+1},\ldots, e_{n}$ are on the right of $L_{x}$ (see Fig. \ref{Figure2-9}).  
Let $E_{1}=\{e_{1},e_{2},\ldots,e_{n_1}\}$ and $E_2=\{e_{n_1+1},\ldots,e_{n}\}$.
Since $C$ is an $(M_{1},M_{2})$-alternating cycle,
we have  $E_1\subseteq M_{1}$ or $M_{2}$.

\begin{figure}[htbp]
\centering
\includegraphics[scale=0.8]{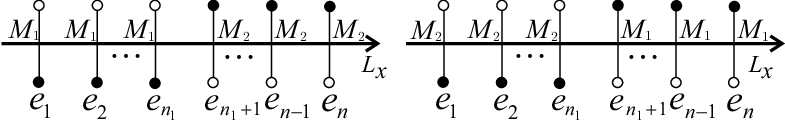}
\caption{\label{Figure2-9} Illustration for two cases of the proof of Lemma \ref{one cycle}(ii).}
\end{figure}
If $E_1\subseteq M_{1}$,
then $E_2\subseteq M_{2}$.
Thus ${\rm flow}_{L_{x}}(I_{i},M_{1})-
{\rm flow}_{L_{x}}(I_{i},M_{2})=1$ $(1\leq i\leq n_{1})$
and ${\rm flow}_{L_{x}}(I_{i},M_{1})-
{\rm flow}_{L_{x}}(I_{i},M_{2})=0-(-1)=1$ $(n_{1}+1\leq i\leq n)$. So $$\sum\limits_{i=1}^{n}({\rm flow}_{L_{x}}(I_{i},M_{1})-
{\rm flow}_{L_{x}}(I_{i},M_{2}))=n.$$

\noindent If $E_1\subseteq M_{2}$,
then $E_2\subseteq M_{1}$. Similarly we have
$$\sum\limits_{i=1}^{n}({\rm flow}_{L_{x}}(I_{i},M_{1})-
{\rm flow}_{L_{x}}(I_{i},M_{2}))=-n.$$
In short, Eq. (\ref{C_{1}}) holds from Eq. (\ref{redu}) and the above two equalities.
\end{proof}

\begin{figure}[htbp]
\centering
\includegraphics{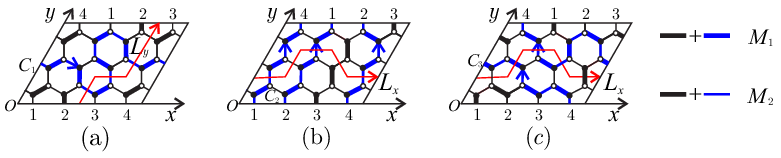}
\caption{\label{Figure2-10} Three $(M_1,M_2)$-alternating cycles (blue) of toroidal polyhex $H(4,3,1)$.}
\end{figure}
\begin{ex}
{\rm For perfect matchings  $M_{1}$, $M_{2}$ of   toroidal polyhex $H(4,3,1)$ as shown  in Fig. \ref{Figure2-10}, $(M_1,M_2)$-alternating cycle $C_1$  in Fig. \ref{Figure2-10}(a) has a winding vector (1,0) and  ${\rm flow}_{L_{y}}(C_{1},M_{1})-{\rm flow}_{L_{y}}(C_{1},M_{2})=-1$; the  cycle $C_{2}$ in Fig. \ref{Figure2-10}(b) has a winding vector $(-1,3)$ and
 ${\rm flow}_{L_{x}}(C_{2},M_{1})-{\rm flow}_{L_{x}}(C_{2},M_{2})=3$;
the  cycle $C_{3}$ in Fig. \ref{Figure2-10}(c) has a winding vector $(-1,1)$
and ${\rm flow}_{L_{x}}(C_{3},M_{1})-{\rm flow}_{L_{x}}(C_{3},M_{2})=1$.}
\end{ex}

An $M$-alternating non-contractible cycle $C$ with an orientation in $H$  is said to be \emph{proper (improper)} if every edge of $C$ belonging to $M$ goes from the white (black) end-vertex to the black (white) end-vertex along the orientation of $C$.

\begin{lem}\label{two non-contractible cycle}
Let $M_{1}$, $M_{2}$ be any two perfect matchings of a toroidal polyhex $H$ and $C_{1}$, $C_{2}$ two disjoint  $(M_{1}, M_{2})$-alternating cycles.  Then one of $C_{1}$ and $C_{2}$ is a proper $M_{1}$-alternating cycle and the other one is improper if
(i) both $C_{1}$ and $C_{2}$ have the same winding vector $(1,0)$  and
 ${\rm flow}_{L_{y}}(C_{1}\cup C_{2},M_{1})={\rm flow}_{L_{y}}(C_{1}\cup C_{2},M_{2})$, or (ii)  both $C_{1}$ and $C_{2}$ have the same winding vector $(m,n)$, $n\geq1$
and ${\rm flow}_{L_{x}}(C_{1}\cup C_{2},M_{1})={\rm flow}_{L_{x}}(C_{1}\cup C_{2},M_{2})$.

\end{lem}
\begin{proof}
Suppose to the contrary that $C_{1}$ and
$C_{2}$ are proper $M_{1}$-alternating cycles
(the proof for improper case is similar).

If the condition (i) holds,  then  ${\rm flow}_{L_{y}}(C_{i},M_{1})-{\rm flow}_{L_{y}}(C_{i},M_{2})=\pm 1$ by  Lemma \ref{one cycle}(i) and its proof, which is determined by the last edge $e_{i}$ in $L_{y}^*\cap E(C_{i})$, $i=1,2$.
If the white end-vertices of $e_{1}$, $e_{2}$ are on the same side of $L_{y}$, say the left side, then $\{e_{1},e_{2}\}\subseteq M_{1}$ since $C_{1}$ and $C_{2}$ have the same winding vector (direction) and are proper $M_{1}$-alternating cycles. Then ${\rm flow}_{L_{y}}(C_{i},M_{1})-{\rm flow}_{L_{y}}(C_{i},M_{2})=1$, $i=1,2$,
so ${\rm flow}_{L_{y}}(C_{1}\cup C_{2},M_{1})-{\rm flow}_{L_{y}}(C_{1}\cup C_{2},M_{2})=2$, a contradiction.
If the white end-vertices of $e_{1}$, $e_{2}$ are on the different sides of $L_{y}$,  without loss of generality  assume that the white end-vertex of $e_{1}$ (resp. $e_{2}$) is on the left (resp. right) side of $L_{y}$. Then 
$e_{1}\in M_{1}$, $e_{2}\in M_{2}$.  Hence ${\rm flow}_{L_{y}}(C_{i},M_{1})-{\rm flow}_{L_{y}}(C_{i},M_{2})=1$, $i=1,2$, and ${\rm flow}_{L_{y}}(C_{1}\cup C_{2},M_{1})-{\rm flow}_{L_{y}}(C_{1}\cup C_{2},M_{2})=2$, a contradiction.

Now suppose that the condition (ii) holds. By  Lemma \ref{one cycle}(ii) with its  proof, for $C=C_1$ or $C_2$  we know that ${\rm flow}_{L_{x}}(C,M_{1})-
{\rm flow}_{L_{x}}(C,M_{2})=\pm n$, which is determined by the last an edge $e_{j}$ of $L_{x}^{*}\cap E(I_{j})$ ($1\leq j\leq n$). Let $E_{1}=\{e_{1},e_{2},\ldots,e_{n}\}$.
Without loss of generality,
suppose that the first $n_{1}$
 edges of $E_{1}$ have the white end-vertices  on the left side of $L_{x}$ and the next  $n-n_{1}$  edges of $E_{1}$ have the white end-vertices
on the right side. Then  $\{e_1,\ldots,e_{n_1}\}\subseteq M_{1}$,
and  $\{e_{n_1+1},\ldots,e_{n}\}\subseteq M_{2}$.
 We can see that
${\rm flow}_{L_{x}}(C,M_{1})-
{\rm flow}_{L_{x}}(C,M_{2})=n_{1}-(-(n-n_{1}))=n$.
Hence ${\rm flow}_{L_{x}}(C_{1}\cup C_{2},M_{1})-
{\rm flow}_{L_{x}}(C_{1}\cup C_{2},M_{2})=2n$,
a contradiction.
\end{proof}

\noindent{\bf Proof of Theorem \ref{main result}.}
{\it Necessity:} Assume that $M_{1}$ and $M_{2}$ lie in the same connected component of the resonance graph $R_6(H)$. By Lemma \ref{necessity},
${\rm flow}_{L_{j}}(H, M_{1})$=${\rm flow}_{L_{j}}(H, M_{2})$ $(j=x,y)$;  by Lemma \ref{same ladder}, $M_{1}$ and $M_{2}$ have precisely the same ladders.

{\it Sufficiency:}  Assume that ${\rm flow}_{L_{j}}(H, M_{1})$=${\rm flow}_{L_{j}}(H, M_{2})$ $(j=x,y)$ and $M_{1}$ and $M_{2}$ have precisely the same ladders.
Let $\mathcal{C}=M_{1}\triangle M_{2}=\cup_{i=1}^k C_{i}$, $k\geq0$, where the $C_{i}$'s are  disjoint $(M_1,M_2)$-alternating cycles.
We will show that $R_{6}(H)$ has a path between $M_{1}$ and $M_{2}$ by induction on the number $k$.
If $k=0$, it is trivial.
Next, let $k\geq 1$.
If $\mathcal{C}$ contains a contractible cycle $C$, then $C$ bounds a hexagonal system. By using the induction hypothesis and the same argument as in the proof of Theorem \ref{N-flow} we can show  the result.
Therefore, suppose that $\mathcal C$ has no contractible cycles. Since ${\rm flow}_{L_{j}}(H, M_{1})={\rm flow}_{L_{j}}(H, M_{2})$ $(j=x,y)$, ${\rm flow}_{L_{j}}(\mathcal{C}, M_{1})
={\rm flow}_{L_{j}}(\mathcal{C},M_{2})$, which implies

\begin{equation}\label{C_{i}}
    \sum\limits_{i=1}^{k}({\rm flow}_{L_{j}}(C_i,M_{1})-
    {\rm flow}_{L_{j}}(C_i,M_{2}))=0.
\end{equation}

By Theorem \ref{two types}, all cycles of $\mathcal{C}$ can have the same winding vector.

{\bf{ Case 1.}}  All cycles of $\mathcal C$  have a same winding vector  $(1,0)$.

By Lemma \ref{one cycle}(i), $|{\rm flow}_{L_{y}}(C_i,M_{1})-{\rm flow}_{L_{y}}(C_i,M_{2})|=1$ for $1\leq i\leq k$.
Therefore, Eq. (\ref{C_{i}}) implies that $k$ is even and a half of the cycles of $C_{1},\ldots,C_{k}$ satisfy
${\rm flow}_{L_{y}}(C_{i},M_{1})-{\rm flow}_{L_{y}}(C_{i},M_{2})=1$,
the other half of cycles satisfy
${\rm flow}_{L_{y}}(C_{i},M_{1})-{\rm flow}_{L_{y}}(C_{i},M_{2})=-1$.
By Lemma \ref{two non-contractible cycle}, a half of the cycles of $C_{1},\ldots,C_{k}$ are proper $M_{1}$-alternating cycles and the other half of cycles are improper $M_{1}$-alternating cycles.  Take a proper and an improper $M_1$-alternating cycles $C_1$ and  $C_2$ in $\mathcal C$.

{\bf{ Case 2.}} All cycles of $\mathcal C$ have a  same winding vector  $(m,n)$, $n\geq 1$.

By Lemma \ref{one cycle}(ii), $|{\rm flow}_{L_{x}}(C_i,M_{1})-{\rm flow}_{L_{x}}(C_i,M_{2})|= n$ for $1\leq i\leq k$. So Eq. (\ref{C_{i}}) implies that $k$ is even
and a half of the cycles of $C_{1},\ldots,C_{k}$ satisfy
${\rm flow}_{L_{x}}(C_{i},M_{1})-{\rm flow}_{L_{x}}(C_{i},M_{2})=n$,
the other half of the cycles satisfy
${\rm flow}_{L_{x}}(C_{i},M_{1})-{\rm flow}_{L_{x}}(C_{i},M_{2})=-n$.
Also take a proper and an improper $M_1$-alternating cycles $C_1$ and  $C_2$ in $\mathcal C$.

In the above two cases, $C_{1}$ and $C_{2}$ have the same winding vector. Hence  they are freely homotopic non-contractible cycles. By Lemma \ref{cylinder},  $C_{1}$ and $C_{2}$ bound a cylinder in $H$. The subgraph consisting of the hexagons contained in this cylinder region forms a nanotube $N$ whose  boundary cycles are $C_1$ and $C_2$.

Let $M_{i}^{'}=M_{i}|_{N}$ for $i=1,2$.  Since $M_{1}$ and $M_{2}$ have precisely the same ladders, $M_{1}'$ and $M_{2}'$ have precisely the same ladders in $N$. By Lemma \ref{elementary nanotube}, $N$ is elementary. Let $M_3=M_1\triangle E(C_1)\triangle E(C_2)$ and $M_{3}^{'}=M_{3}|_{N}$.
Then $M_{1}^{'}\triangle M_{3}^{'}=E(C_1)\cup (C_2)$. We can show that $R_6(H)$ has a path from $M_1$ to $M_3$ and a path from $M_3$ to $M_2$ to complete the proof by an analogous to the final paragraph in the proof of Theorem \ref{N-flow}.
\hfill $\square$

\section{Some consequences}

The proof  of Theorem \ref{main result} implies  immediately the following consequence.
\begin{cor}\label{even}
Let $H$ be a toroidal polyhex and $M$, $M'$  two perfect matchings of $H$.  If $M$ and $M'$
lie in the same connected component of  $R_6(H)$,
then $H[M\triangle M']$ contains an even number of non-contractible cycles, half of which are proper $M$-alternating and the other half are improper $M$-alternating.
\end{cor}
\begin{proof} 
By Theorem \ref{main result}, ${\rm flow}_{L_{j}}(H, M)={\rm flow}_{L_{j}}(H, M')$ $(j=x,y)$ and  $M_{1}$ and $M_{2}$ have precisely the same ladders. Further, we have ${\rm flow}_{L_{j}}(\mathcal{C}, M)={\rm flow}_{L_{j}}(\mathcal{C},M')$  as $M\triangle M'=\mathcal{C}$. By Lemma \ref{one contractible cycle}, for any contractible cycle $C$ of $\mathcal{C}$, $
{\rm flow}_{L_{j}}(C,M)-{\rm flow}_{L_{j}}(C,M')=0.$
For  the non-contractible cycles $C_{1},\ldots,C_{t}$  of $\mathcal{C}$,
\begin{equation}\label{4.6}
    \sum\limits_{i=1}^{t}({\rm flow}_{L_{j}}(C_i,M)-
    {\rm flow}_{L_{j}}(C_i,M'))=0,
\end{equation}
which implies that $t$ is even, and half of the cycles $C_i$ are proper $M$-alternating and the other half are improper $M$-alternating by analogous arguments to Cases 1 and  2 in the proof of Theorem \ref{main result}. \end{proof}

A perfect matching $M$ of a toroidal polyhex (or nanotube) $H$ is called {\it trivial}  if $H$ has  no $M$-alternating hexagons.

\begin{lem}\label{trivial}
Let $M$ be a perfect matching of a toroidal polyhex  $H$. Then the following statements are equivalent.
(i) $M$ contains a ladder; (ii) $M$ consists of ladders; (iii) $M$ is trivial.
\end{lem}
\begin{proof}
(i) $\Rightarrow$ (ii) Suppose that $M$ contains a ladder $Z$. Then $Z$ is an $M$-alternating zigzag cycle. If $Z$ is a III-direction ($a,b$-edge directions) cycle (see Fig. \ref{Figure2-4}), then the edges of $\partial(V(Z))$ are $c$-direction edges and does not belong to $M$.
By Remark \ref{zigzag}, $H$ has $q$ III-direction cycles, say  $C_{1},\ldots,C_{q}$ in turn. 
Without loss of generality, let $Z=C_{1}$.
Then the end-vertices of the edges of $\partial(V(C_{1}))$ not in $V(C_1)$  are covered by $M$, so the two III-direction cycles $C_{2}$ and $C_{q}$ are $M$-alternating cycles, which are  ladders of $M$.  Repeating the above process for $C_{i}$ ($i=2,\ldots,q-1$) in turn, we can obtain that all $C_i$ are ladders of $M$ and none of  $c$-direction edges of $H$  belong to $M$ (For an example, see Fig. \ref{Figure2-4}).
Thus, $M$ consists of ladders. If $Z$ is an I or II-direction  cycle (see Fig. \ref{Figure2-4}), similarly we can obtain the result.

(ii) $\Rightarrow$ (iii) Suppose that $M$ consists of ladders. Then $M$ has only two edge-direction edges.
Thus, $H$ has no $M$-alternating hexagons and $M$ is trivial.

(iii) $\Rightarrow$ (i) Suppose that $M$ is trivial.  That is, $H$ has no $M$-alternating hexagons. If $M$ contains edges of at most two edge-directions, then $M$ consists of ladders and we are done. Otherwise, take
an $a$-direction edge in $M$ and  a  III-direction cycle $C$ containing it.
If all vertices of $C$ are covered by the $a$-direction edges in $M$, then $C$ is a ladder of $M$ and we are done. Otherwise,  there exist a vertex $v$ of $C$ incident with a $c$-direction edge $e_{1}$ of $M$ and an $a$-direction edge $e_2$ in both $M$ and $C$ that is incident to a neighbor of $v$. Let $s_{1}$ be the hexagon of $H$ containing $e_{1}$ and  $e_{2}$. Then the $b$-direction edge $e_{3}$ of $s_{1}$ not adjacent to both $e_{1}$ and $e_{2}$  do not belong to $M$. Otherwise, $s_{1}$ is an $M$-alternating hexagon, a contradiction. Hence $M$ contains an $a$-direction edge $e_{5}$ and a $c$-direction edge $e_{4}$ of $M$ adjacent to $e_{3}$.   Let $s_{2}$ be a hexagon of $H$ containing $e_{4}$ and $e_{5}$.  Repeating the procedure, we eventually obtain an $M$-alternating II-direction cycle. Thus, $M$ contains a ladder.
\end{proof}

Each trivial perfect matching of a toroidal polyhex $H$ is a singleton of its resonance graph. Therefore, if two distinct perfect matchings of a toroidal polyhex $H$ lie in the same connected component of its resonance graph, then both must be nontrivial. Hence by adding such a natural condition we obtain a stronger result than Theorem \ref{main result} from Lemma \ref{trivial}.
\begin{thm}\label{stronger}
Let $H$ be a toroidal polyhex with two cut cycles $L_{x}$ and $L_{y}$. Two nontrivial perfect matchings $M_1$ and $M_2$ of $H$ lie in the same connected component of the resonance graph of  $H$ if and only if their flows across the cut cycles $L_{j}$ ($j=x,y$) are equal.
\end{thm}
\begin{proof}
The necessity follows by  Theorem \ref{main result}. For the sufficiency,
since $M_{1}$ and $M_{2}$ are nontrivial perfect matching of a toroidal polyhex $H$, by Lemma \ref{trivial},  $M_{1}$ and $M_{2}$ contain no ladders, which imply that $M_{1}$ and $M_{2}$ have precisely the same ladders. By the sufficiency of Theorem \ref{main result}, $R(H)$ has a path between $M_1$ and $M_2$.
\end{proof}

\begin{rem}\label{PMnumber}{\rm By Lemma \ref{trivial}  and Remark \ref{zigzag} we have that a toroidal polyhex $H(p,q,t)$ has exactly $2^{q}+2^{\frac{p}{g}}+2^{\frac{p}{g'}}-3$ trivial perfect matchings, where $g$ (resp. $g'$) is the smallest positive integer satisfying $gt\equiv0 (\text{mod } p)$ $($resp. $g'(q + t)\equiv0 (\text{mod } p))$.}\end{rem}
\begin{figure}
\centering
\includegraphics[scale=0.7]{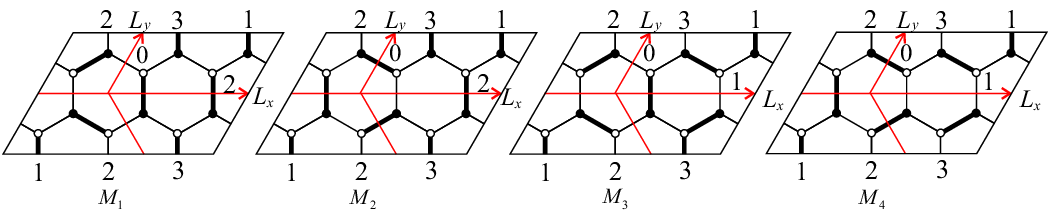}
\caption{\label{Figure2-11} Toroidal polyhex $H(3,2,2)$ and its four perfect matchings. }
\end{figure}
\begin{ex}
{\rm Let $H=H(3,2,2)$ be a toroidal polyhex and $M_{1}$, $M_{2}, M_{3}, M_{4}$ be four nontrivial perfect matchings of $H$ as shown in Fig. \ref{Figure2-11}.
Since
${\rm flow}_{L_{x}}(H, M_{1})={\rm flow}_{L_{x}}(H, M_{2})=2$ and ${\rm flow}_{L_{y}}(H, M_{1})={\rm flow}_{L_{y}}(H, M_{2})=0$, by  Theorem \ref{stronger}
we have that  $M_{1}$ and $M_{2}$  lie in the same connected component of  $R_6(H)$; Since ${\rm flow}_{L_{x}}(H, M_{3})={\rm flow}_{L_{x}}(H, M_{4})=1$ and ${\rm flow}_{L_{y}}(H, M_{3})={\rm flow}_{L_{y}}(H, M_{4})=0$,  $M_{3}$ and $M_{4}$ lie in the same connected component of  $R_6(H)$. Since ${\rm flow}_{L_{x}}(H, M_{1})\neq{\rm flow}_{L_{x}}(H, M_{3})$, $M_{1}$ and  $M_{3}$  lie in different connected components of  $R_6(H)$.} \end{ex}

\begin{rem}
{\rm The sufficiency of Theorem \ref{stronger} no longer holds for nanotubes.  The reason is that  Lemma \ref{trivial} does not hold for nanotubes. For example, given a (3, 0)-nanotube $N$ with  two perfect matchings $M_{1}$ and $M_{2}$ as shown in Fig. \ref{Figure2-12}. Then $M_{1}$ and $M_{2}$ are nontrivial perfect matchings since both $M_{1}$ and $M_{2}$-alternating hexagons exist, and  ${\rm flow}_{L}(N, M_{1})={\rm flow}_{L}(N, M_{2})=0$. However,  $M_{1}$ and $M_{2}$ contain different ladders, so  $M_{1}$ and  $M_{2}$ do not lie in the same connected component of  $R_6(N)$ by Theorem \ref{N-flow}.}
\end{rem}

\begin{figure}[ht]
\centering
\includegraphics[scale=0.7]{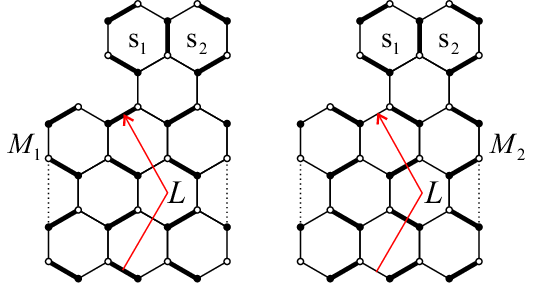}
\caption{\label{Figure2-12} A (3, 0)-nanotube and its two nontrivial perfect matchings.}
\end{figure}
\vspace{5pt}
\noindent{\textbf{Declarations}}\\
\textbf{Conflict of interest} The authors declare no conflict of interest.
%


\begin{thebibliography}{99}
\small \setlength{\itemsep}{-.4mm}


\bibitem{BCTZ25}
S. Brezovnik, Z. Che, N. Tratnik,  P. \v{Z}igert Pleter\v{s}ek, Resonance graphs of plane bipartite graphs as daisy cubes, Discrete Appl. Math. 366 (2025) 75-85.

\bibitem{Che18}
 Z. Che, Structural properties of resonance graphs of plane elementary bipartite graphs,
Discrete Appl. Math. 247 (2018) 102-110.


\bibitem{DP26} T. Do\v{s}li\'{c}, L. Podrug, Metallic cubes are the resonance graphs of phenylenes, Discrete Appl. Math. 378 (2026) 358-365.


\bibitem{EL21} J. Erickson, P. Lin, A toroidal Maxwell–Cremona–Delaunay correspondence, J. Comput. Geometry 12(2) (2021) 55-85.


\bibitem{Four03} J.C. Fournier, Combinatorics of perfect matchings in plane bipartite graphs and application to tilings, Theoret. Comput. Sci. 303 (2003) 333-351.

\bibitem{Gree09} S. Greenberg,  D. Randall, Convergence rates of Markov chains for some self-assembly and non-saturated Ising models, Theoret. Comput. Sci. 410 (2009) 1417-1427.

\bibitem{WGrun82} W. Gr\"{u}ndler, Signifilkante Elektronenstrukturen fur benzenoide Kohlenwasserstoffe. Wiss. Z. Univ. Halle. 31 (1982) 97-116.

\bibitem{Hatc01} A. Hatcher, Algebraic Topology, Amer. Math. Soc. Providence, RI, 2022.

\bibitem{LZhang03} P.C.B. Lam, H. Zhang, A distributive lattice on the set of perfect matchings of a plane bipartite graph, Order 20 (2003) 13-29.

\bibitem{Levi63} H.I. Levine, Homotopic curves on surfaces, Proc. Am. Math. Soc. 14(6) (1963) 986-990.


\bibitem{LZhan26} L. Liang, H, Zhang, Resonance graphs of coronoid systems and nanotubes, Discrete Appl. Math. 395 (2026) 443-455.

\bibitem{LWLZ26} Q. Liu, J. Wang, C. Li, H. Zhang, Components of flip graphs of domino tilings in quadriculated cylinder and torus, Appl. Math. Comput. 510 (2026) 129697.


\bibitem{LZZI25} Q. Liu, Y. Zhang, H. Zhang, Components of domino tilings under flips in quadriculated tori, Discrete Math. 348 (2025) 114396.

\bibitem{Lovp86} L. Lov\'asz, M. D. Plummer,  Matching Theory, North-Holland, Amsterdam, 1986.


\bibitem{STCR95} N.C. Saldanha, C. Tomei, M.A. Casarin, D. Romualdo, Spaces of domino tilings,
Discrete Comput. Geom. 14 (1995) 207-233.

\bibitem{Scha76}
J. A. Schafer, Representing homology classes on surfaces, Canad. Math. Bull. 19 (3) (1976) 373-374.

\bibitem{Schr93} A. Schrijver, Graphs on the torus and geometry of numbers, J. Combin. Theory Ser. B. 58(1) (1993) 147-158.


\bibitem{SLZh05} W.C. Shiu, P.C.B. Lam, H. Zhang, $k$-resonance in toroidal polyhexes, J. Math. Chem. 38(4) (2005) 451-466.


\bibitem{SHZ96}
H. Sachs, P. Hansen, M. Zheng, Kekul\'{e} count in tubular hydrocarbons, MATCH
Commun. Math. Comput. Chem. 33 (1996) 169-241.


\bibitem{T90} W.P. Thurston, Conway’s tiling groups, Amer. Math. Monthly 97(8) (1990) 757-773.


\bibitem{TYra23} N. Tratnik, D. Ye, Resonance graphs on perfect matchings of graphs on surfaces, Graphs Combin. 39(4) (2023) 15pp.

\bibitem{TZra15} N. Tratnik, P. \v{Z}igert Pleter\v{s}ek, Some properties of carbon nanotubes and their resonance graphs, MATCH Commun. Math. Comput. Chem.  74 (2015) 175-186.

\bibitem{TZra16} N. Tratnik, P. \v{Z}igert Pleter\v{s}ek, Resonance graphs of fullerenes, Ars Math. Contemp. 11 (2016) 425-435.

\bibitem{TZ16} N. Tratnik, P. $\rm{\check{Z}}$igert Pleter$\rm{\check{s}}$ek, Distributive lattice structure on the set of perfect matchings of carbon nanotubes, J. Math. Chem. 54(6) (2016) 1296-1305.


\bibitem{WYZ08} H. Wang, D. Ye, H. Zhang, The forcing number of toroidal polyhexes, J. Math. Chem. 43(2) (2008) 457-475.


\bibitem{ZGC88} F. Zhang, X. Guo, R. Chen, Z-transformation graphs of perfect matchings of hexagonal systems, Discrete Math. 72 (1988) 405-415.

\bibitem{ZLs08}
H. Zhang, P.C.B. Lam, W.C. Shiu, Resonance graphs and a binary coding for the 1-factors of benzenoid systems, SIAM J. Discrete Math. 22 (2008) 971-984.



\bibitem{Zhang06} H. Zhang, Z-transformation graphs of perfect matchings of plane bipartite graphs: a survey, MATCH Commun. Math. Comput. Chem. 56(3) (2006) 457-476.


\bibitem{ZY08} H. Zhang, D. Ye, $k$-resonant toroidal polyhexes, J. Math. Chem. 44(1) (2008) 270-285.

\bibitem{ZZ2000} H. Zhang, F. Zhang, Plane elementary bipartite graphs, Discrete Appl. Math. 105 (2000) 291-311.

\end{thebibliography}
\end{document}